\documentclass[11pt,reqno]{amsart}

\date{\today}

\usepackage[final]{graphicx}
\usepackage{amsfonts}
\usepackage{amsmath}
\usepackage{amssymb}
\usepackage{amsthm}
\usepackage{mathrsfs}
\usepackage[colorlinks=true,linkcolor=blue,citecolor=blue]{hyperref}
\usepackage{bbm}
\usepackage{setspace}

\newtheorem{theorem}{Theorem}
\newtheorem{lemma}[theorem]{Lemma}

\makeatletter\@namedef{subjclassname@2020}{\textup{2020} Mathematics Subject Classification}
\newcommand{\msc}[1]{\href{https://zbmath.org/classification/?q=cc:#1}{#1}}
\newcommand{\Email}[1]{\rm{\it E-mail:}\/~\href{mailto:#1}{\textsf{#1}}}
\newcommand{\Emails}[2]{\rm{\it E-mails:}\/~\href{mailto:#1}{\textsf{#1}}, \href{mailto:#2}{\textsf{#2}}}

\newcommand{\1}{\mathbbm1}
\newcommand{\R}{{\mathord{\mathbb R}}}

\newcommand{\N}{{\mathord{\mathbb N}}}

\newcommand{\Sp}{{\mathord{\mathbb S}}}
\newcommand{\Sph}{\mathbb{S}}

\newcommand{\bb}{g}
\renewcommand{\(}{\left(}
\renewcommand{\)}{\right)}
\newcommand{\ird}[1]{\int_{\R^d}{#1}\,dx}
\newcommand{\isd}[1]{\int_{\Sph^d}{#1}\,d\mu_d}
\newcommand{\nrm}[2]{\left\|{#1}\right\|_{\mathrm L^{#2}(\R^d)}}
\newcommand{\nrmS}[2]{\left\|{#1}\right\|_{\mathrm L^{#2}(\mathbb S^d)}}
\newcommand{\Deficit}{\boldsymbol{\delta}}

\newcommand{\Imu}{\mathscr I}
\newcommand{\be}[1]{\begin{equation}\label{#1}}
\newcommand{\ee}{\end{equation}}
\newcommand{\mmu}{\boldsymbol{\mu}}
\newcommand{\irdmu}[1]{\int_{\R^d}{#1}\,d\mmu_d}
\newcommand{\irNg}[1]{\int_{\R^N}{#1}\,d\gamma}
\newcommand{\e}{\mathsf e}

\renewcommand{\i}{\mathsf i}

\newcommand{\Hdot}{\dot{\mathrm H}^1(\R^d)}
\newcommand{\Emph}[1]{\textit{\textbf{#1}}}
\newcommand{\bfA}{\mathrm A}

\newcommand{\nc}{\normalcolor}

\usepackage{etoolbox}
\newcounter{taggedeq}
\pretocmd{\equation}{\stepcounter{taggedeq}}{}{}

\begin{document}

\title[Improvements and stability for some interpolation inequalities]{A short review on Improvements and stability\\ for some interpolation inequalities}

\author[J.~Dolbeault]{Jean Dolbeault}
\author[M.~J.~Esteban]{Maria J.~Esteban}
\address[J.~Dolbeault, M.~J.~Esteban]{CEREMADE (CNRS UMR No.~7534), PSL University, Universit\'e Paris-Dauphine, Place de Lattre de Tassigny, 75775 Paris 16, France. \newline\Emails{dolbeaul@ceremade.dauphine.fr}{esteban@ceremade.dauphine.fr}}

\author[A.~Figalli]{Alessio Figalli}
\address[A.~Figalli]{Mathematics Department, ETH Z\"urich, Ramistrasse 101, 8092 Z\"urich, Switzerland.\newline \Email{alessio.figalli@math.ethz.ch}}

\author[R.~L.~Frank]{Rupert L.~Frank}
\address[R.~L.~Frank]{Department of Mathematics, LMU Munich, Theresienstr. 39, 80333 M\"unchen, Germany, and Munich Center for Quantum Science and Technology, Schellingstr.~4, 80799 M\"unchen, Germany, and Mathematics 253-37, Caltech, Pasadena, CA 91125, USA.\newline
\Email{r.frank@lmu.de}}

\author[M.~Loss]{Michael Loss}
\address[M.~Loss]{School of Mathematics, Georgia Institute of Technology Atlanta, GA 30332, United States of America. \Email{loss@math.gatech.edu}}


\begin{abstract}
In this paper, we present recent stability results with explicit and dimensionally sharp constants and optimal norms for the Sobolev inequality and for the Gaussian logarithmic Sobolev inequality obtained by the authors in \cite{DEFFL}. The stability for the Gaussian logarithmic Sobolev inequality was obtained as a byproduct of the stability for the Sobolev inequality. Here we give a new, direct, alternative proof. We also discuss improved versions of interpolation inequalities based on the \it carr\'e du champ \rm method.\vspace*{-0.5cm}
\end{abstract}



\thanks{\scriptsize \copyright~2024 by the authors. Reproduction of this article by any means permitted for noncommercial purposes. \hbox{\href{https://creativecommons.org/licenses/by/4.0/legalcode}{CC-BY 4.0}}}

\maketitle
\thispagestyle{empty}

\section{Introduction and main results}\label{Sec:intro}

Let us assume that $d\ge3$ without further notice. The classical \it Sobolev inequality \rm on $\R^d$, can be written as follows:
\be{Sob}\tag{S}
\Vert \nabla f \Vert^2_{\mathrm L^2(\R^d)}\ge S_d\,\Vert f \Vert^2_{\mathrm L^{2^*}(\R^d)}\quad\forall\,f\in\dot{\mathrm H}^1(\R^d)\,,
\ee
where $2^*=\frac{2\,d}{d-2}$ is the Sobolev exponent, $S_d = \tfrac14\,d\,(d-2)\,| \mathbb S^d|^{2/d}$ is the sharp Sobolev constant, and~$|\mathbb S^d|$ denotes the volume of the unit sphere $\mathbb S^d\subset\R^{d+1}$.
Here $\dot{\mathrm H}^1(\R^d)$ denotes the closure of $C^\infty_c(\R^d)$ with respect to the seminorm $| f|_{\dot{\mathrm H}^1(\R^d)}:=\Vert \nabla f \Vert^2_{\mathrm L^2(\R^d)}$. As proved in \cite{Aubin,Talenti} (see also \cite{Rodemich, Rosen}), equality in~\eqref{Sob} holds if $f$ is one of the \it Aubin-Talenti \rm functions, that is, one of the functions belonging to the $(d+2)$-dimensional \it Aubin-Talenti manifold\rm
\[
\mathcal M:=\left\{g_{a,b,c}\,:\,(a,b,c)\in(0,+\infty)\times\R^d\times\R\right\}\quad\mbox{with}\quad g_{a,b,c}(x) := c\(a+|x-b|^2\)^{-\frac{d-2}2}\,.
\]
In fact, there is equality in \eqref{Sob} if and only if $f$ is in $\mathcal M$ according to~\cite{Lieb,GidasNiNirenberg,CaffarelliGidasSpruck}.

Applying the inverse of the stereographic projection and integrating on the sphere $\mathbb S^d$ with respect to the uniform probability measure $d\mu_d$, one can rewrite the above inequality as
\be{Sob-sphere}
\nrmS{\nabla u}2^2\ge\tfrac14\,d\,(d-2)\(\nrmS u{2^*}^2-\nrmS u2^2\)\quad \forall\, u\in \mathrm H^1(\mathbb S^d)
\ee
and state, equivalently, that the only functions in $\mathrm H^1(\mathbb S^d)$ for which there is equality in~\eqref{Sob-sphere} are the functions $\omega\mapsto G_{b,c}(\omega):=c\,(1+b\cdot \omega)^{-(d-2)/2}$ where $b\in B_1:=\{b\in\R^{d+1}\,:\,|b|<1\}$ and $c\in\R$ are constants. One can embed \eqref{Sob-sphere} in the following family of \it Gagliardo-Nirenberg-Sobolev inequalities:\rm
\be{GNS}\tag{GNS}
\nrmS{\nabla u}2^2\ge\frac d{p-2}\(\nrmS up^2-\nrmS u2^2\)\quad \forall\, u\in \mathrm H^1(\mathbb S^d)\,,\; \forall\, p\in(1, 2)\cup (2, 2^*]\,.
\ee
As proved by Bidaut-V\'eron and V\'eron~\cite{BV-V} and Beckner~\cite{MR1230930} for $p>2$ (see also \cite{MR772092, Bakry-Emery85, MR808640} if $p\le 2^\#$), the constant $d/(p-2)$ is the best possible constant in \eqref{GNS}. Here $2^\#:=(2\,d^2+1)/(d-1)^2<2^*$ denotes the Bakry-Emery exponent. If $p<2^*$, the only optimizers in $\mathrm H^1(\mathbb S^d)$ for~\eqref{GNS} are the constant functions. The \it carr\'e du champ \rm method used in~\cite{MR772092, Bakry-Emery85, MR808640} relies on the linear heat equation, which induces the limitation $p\le 2^\#$. This limitation is not technical as shown in \cite{MR3640894}. The whole range $p\in[1, 2)\cup (2, 2^*]$ was covered using a \it carr\'e du champ \rm method based on nonlinear diffusion equations by Demange and Dolbeault-Esteban-Loss in \cite{MR2381156, 1302}. The case $p=2$ (that can be obtained from \eqref{GNS} by taking the limit $p\to 2$) is the \it logarithmic Sobolev inequality \rm
\be{Ineq:logSob}\tag{LSI}
\nrmS{\nabla u}2^2\ge\frac d2\,\isd{|u|^2\,\ln\(\frac{|u|^2}{\nrmS u2^2}\)}\quad\forall\,u\in\mathrm H^1(\mathbb S^d)\setminus\{0\}\,,
\ee
where the only optimizers are the constant functions. Inequality~\eqref{Ineq:logSob} can also be proved by the \it carr\'e du champ \rm method based on the heat equation. 

A remarkable feature of the \it carr\'e du champ \rm method is that additional terms appear in the computations. This has been used in \cite{MR3177759} to establish, for all $p\in(2,2^*)$, the \it improved interpolation inequalities \rm
\be{improved-sub}
\nrmS{\nabla u}2^2-\frac d{p-2}\(\nrmS up^2-\nrmS u2^2\)\ge\nrmS u2^2\,\Psi_{d,p}\!\(\frac{\nrmS{\nabla u}2^2}{\nrmS u2^2}\)\quad\forall\,u\in\mathrm H^1(\mathbb S^d)\setminus\{0\}
\ee
for some convex function $\Psi_{d,p}$ such that $\Psi_{d,p}(0)=\Psi_{d,p}'(0)=0$ and $\Psi_{d,p}(s)>0$ if $s>0$. Inequality~\eqref{improved-sub} is obviously stronger than \eqref{GNS} and can be used to prove that the equality case in \eqref{GNS} is realized only by constant functions. If we test this inequality with $u_\varepsilon(\omega)=1+\varepsilon\,\mathsf b\cdot\omega$ for a given $\mathsf b\in B_1$, an elementary computation shows that there is a cancellation of the $O(\varepsilon^2)$ terms as $\varepsilon\to0$ and the first non-zero term is of the order of $O(\varepsilon^4)$.

If we denote by $\Pi_1$ the $\mathrm L^2(\mathbb S^d)$ projection on the space generated by the coordinate functions $\omega_i$ with $i=1$, $2$, \ldots, $d$, we learn from \cite[Theorem~6]{Brigati_2023} that there is an explicit constant $\kappa=\kappa(p,d)\in(0,1)$, depending on $d$ and $p\in[1,2)\cup(2,2^*)$ such that, for all $u\in\mathrm H^1(\mathbb S^d)$,
\begin{equation}\label{eq:stabgns4}
\nrmS{\nabla u}2^2-\frac d{p-2}\(\nrmS up^2-\nrmS u2^2\)\ge\kappa\(\tfrac{\nrmS{\nabla\Pi_1u}2^4}{\nrmS{\nabla u}2^2+\nrmS u2^2}+\nrmS{\nabla(\mathrm{Id}-\Pi_1)\kern 1pt u}2^2\)\,.
\end{equation}
The proof relies on \eqref{improved-sub} and on a decomposition in spherical harmonics directly inspired by~\cite{Frank_2022}. Inequality \eqref{eq:stabgns4} improves over \eqref{GNS}, but this improvement degenerates as $p\to 2^*$. Indeed, any constant $\kappa(p,d)$ for which \eqref{eq:stabgns4} holds necessarily satisfies $\lim_{p\to2^*}\kappa(p,d)=0$, for otherwise we had an inequality at $p=2^*$ whose right side vanishes only for constants while, according to our discussion of \eqref{Sob-sphere}, the left side vanishes for all function $G_{b,c}$.

Let us notice that Inequality~\eqref{improved-sub} can take various forms. For instance, if $p\in(2,2^\#)$ and $1/\theta=1+(\frac{d-1}{d+2})^2\,(p-1)\,(2^\#-p)/(p-2)$, we read from~\cite[Ineq.~(2.4)]{Dolbeault_2020} that \eqref{improved-sub} takes the form
\[
\nrmS{\nabla u}2^2\ge\frac{d\,\theta}{p-2}\(\nrmS up^{2/\theta}\,\nrmS u2^{2-2/\theta}-\nrmS u2^2\)\quad\forall\,u\in\mathrm H^1(\mathbb S^d)\,,
\]
which is a strict improvement compared to~\eqref{GNS}, as can be seen using H\"older inequalities. The case $p=2$ is also covered, except that in the left-hand side~of \eqref{improved-sub}, the deficit of \eqref{GNS} has to be replaced by the deficit of \eqref{Ineq:logSob}. 

Without entering into the details, let us quote some related results. By a direct variational approach, an improved inequality like~\eqref{improved-sub} is proved in \cite[Theorem~2]{Frank_2022}, in the subcritical range, for some $\Psi_{d,p}(s)\sim s^2$. Using nonlinear flows and appropriate orthogonality constraints, improved inequalities with $\Psi_{d,p}(s)\sim s$ are known from \cite{MR3640894}. Both results are unified in~\cite{Brigati_2023}. Let us finally mention that improved inequalities are proved in \cite{BDNS} with explicit constants, not on $\mathbb S^d$ but on $\R^d$, for the Gagliardo-Nirenberg-Sobolev inequalities using entropy methods and regularization effects for fast diffusion flows, but with some restrictions on the decay of the functions at infinity.

The Gaussian measure on $\R^N$ can be seen as an infinite dimensional limit $d\to+\infty$ of the uniform probability measure on the sphere of dimension $d$ and radius $\sqrt d$ tested against functions depending only on a finite number $N$ of coordinates: see for instance \cite{MR353471,https://doi.org/10.48550/arxiv.2302.03926,DEFFL}. Since $\lim_{d\to+\infty}2^*=2$, it also turns out that~\eqref{Sob} has to be replaced by a Gaussian interpolation inequality as follows. On $\R^N$, with $N\ge1$, let us consider the {\it Gaussian measure} $d\gamma(x) := e^{-\,\pi\,|x|^2}\,dx$. With $\mathrm L^2(\gamma):=\mathrm L^2(\R^N,d\gamma)$, if $\mathrm H^1(\gamma)$ denotes the space of all $u\in \mathrm L^2(\gamma)$ with distributional gradient in $\mathrm L^2(\gamma)$, the \it logarithmic Sobolev inequality \rm is:
\be{LSI}\tag{LSI}
\int_{\R^N} |\nabla u|^2\,d\gamma\ge\pi \int_{\R^N} |u|^2 \ln\( \frac{|u|^2}{\| u \|_{\mathrm L^2(\gamma)}^2}\)\,d\gamma\quad\forall\,u\in\mathrm H^1(\gamma)\setminus\{0\}\,.
\ee
According to the result of Carlen \cite[Theorem~5]{MR1132315}, equality in \eqref{LSI} holds if and only if for some $a\in\R^N$ and $c\in\R$,
\be{eq:optim-logSob}
u(x) = c\,e^{a\cdot x}\,.
\ee
Improved forms of the inequality are also known for instance from \cite{MR3567822,MR3493423, MR4116725,Brigati_2023}, under some restrictions. The article \cite{MR4475270} makes a very interesting connection between deficit estimates for \eqref{Ineq:logSob} (for certain restricted classes of functions) and the Mahler conjecture in convex geometry. A detailed list of earlier results and references can be found in~\cite{MR4305006}.

\medskip Now let us turn our attention to \it stability \rm issues for~\eqref{Sob},~\eqref{GNS} and~\eqref{LSI}.

\smallskip\noindent$\bullet$ \Emph{Stability for the Sobolev inequality}. In~\cite{BrezisLieb} Brezis and Lieb asked the following question:

\smallskip \it (Q) Do there exist constants $\kappa$, $\alpha>0$ such that the \emph{Sobolev deficit} $\Deficit$ controls some distance $\mathbf{d}$ from the Aubin-Talenti manifold $\mathcal M$ according to\rm
\[
\Deficit(f):= \Vert \nabla f \Vert^2_{\mathrm L^2(\R^d)}- S_d\, \Vert f \Vert^2_{\mathrm L^{2^*}(\R^d)} \ge \kappa\,\mathbf d(f,\mathcal M)^\alpha\quad?
\]
The `best possible answer' to this question would involve finding the strongest possible topology to define the distance $\mathbf d$ and the best possible constant $\kappa$ and exponent $\alpha$. The first answer to Brezis and Lieb's question was given by Bianchi and Egnell in~\cite{BianchiEgnell}: there is a constant $C_{d,\rm BE}>0$ such that
\be{eq:bianchi-egnell0}
\Deficit(f)\ge C_{d, \rm BE}\,\inf_{\bb\in\mathcal M}\Vert \nabla f - \nabla\bb \Vert^2_{\mathrm L^2(\R^d)}\quad\forall\,f \in \dot{\mathrm H}^1(\R^d)\,.
\ee
Similar results for other inequalities have been proved using the strategy of Bianchi and Egnell: see, for example, \cite{MR3179693}. The main drawback of this strategy is that no explicit estimate of $C_{d,\rm BE}$ is known nor its dependence on the dimension~$d$. Recently, in \cite{DEFFL}, we proved the following result.
\begin{theorem}[{\cite[Theorem~1.1]{DEFFL}}]\label{main} There is an explicit constant $\beta>0$ such that for any $d\ge3$,
\be{stabgrandd}
\Vert \nabla f \Vert^2_{\mathrm L^2(\R^d)} - S_d\,\Vert f \Vert^2_{\mathrm L^{2^*}(\R^d)}\ge\;\frac\beta d\;\inf_{\bb\in\mathcal M} \Vert \nabla f - \nabla\bb \Vert^2_{\mathrm L^2(\R^d)}\quad\forall\,f \in \dot{\mathrm H}^1(\R^d)\,.
\ee
\end{theorem}
\noindent This result is dimensionally sharp. Indeed, Theorem~\ref{main} can be rewritten as $C_{d,\rm BE}\ge\beta/d$. On the other hand, it was proved implicitly in~\cite{BianchiEgnell} (see also \cite{MR3179693}) that $C_{d,\rm BE}\le4/(4+d)$. This inequality is in fact strict: $C_{d,\rm BE}<4/(4+d)$, according to~\cite{arXiv:2210.08482}, and we learn from \cite{arXiv:2211.14185} that equality in \eqref{eq:bianchi-egnell0}, written with the optimal value of $C_{d,\rm BE}$, is achieved. Hence, Theorem~\ref{main} captures the dimensional behavior of $C_{d,\rm BE}$. For completeness, let us mention the extension of Theorem~\ref{main} in \cite{chen2023stability} to fractional Sobolev inequalities.

Using the inverse stereographic projection, the stability result of Theorem~\ref{main} can be rewritten on~$\mathbb S^d$ for any $u\in\mathrm H^1(\mathbb S^d)$ as
\begin{multline*}
\nrmS{\nabla u}2^2-\tfrac14\,d\,(d-2)\(\nrmS u{2^*}^2-\nrmS u2^2\)\\
\ge\;\frac\beta d\;\inf_{(b,c)\in B_1\times\R}\(\nrmS{\nabla u-\nabla G_{b,c}}2^2+\tfrac14\,d\,(d-2)\,\nrmS{u-G_{b,c}}2^2\)\,.
\end{multline*}

\smallskip\noindent$\bullet$ \Emph{Stability for the Gaussian logarithmic Sobolev inequality}. The interpretation of the Gaussian measure as the limit of uniform probability measures on $d$-dimensional spheres as $d\to+\infty$ and the explicit dimensional dependence of the stability constant of Theorem~\ref{main} provides us with a stability result for \eqref{LSI}.
\begin{theorem}[{\cite[Corollary~1.2]{DEFFL}}]\label{logsob} With $\beta>0$ as in Theorem~\ref{main}, for all $N\in\N$, we have
\be{LSI:stab}
\int_{\R^N} |\nabla u|^2\,d\gamma - \pi \int_{\R^N} |u|^2 \ln\( \frac{|u|^2}{\| u \|_{\mathrm L^2(\gamma)}^2}\)\,d\gamma\ge \frac{\beta\,\pi}2 \inf_{b\in\R^N\!,\,c\in\R} \int_{\R^N} \big|u - c\,e^{b\cdot x}\big|^2\,d\gamma\quad\forall\,u\in\mathrm H^1(\gamma)\,.
\ee
\end{theorem}
\noindent In \cite{DEFFL} this result was obtained as a corollary of Theorem \ref{main}. In this paper we give a new, direct proof which highlights the strategy of \cite{DEFFL} in a slightly simpler setting.

\smallskip In the next sections, we briefly describe the strategies of \cite{MR3177759} and \cite{DEFFL} to prove the improvements in the case of subcritical inequalities (Section \ref{Sec:improv}) and the stability results of Theorem \ref{main} (Section \ref{Sec:stabilitySob}) and Theorem~\ref{logsob} (Section~\ref{Sec:stabilitylogSob}). The new proof of Theorem~\ref{logsob} is given in Section~\ref{sec:logsob}.

\section{Improved interpolation inequalities}\label{Sec:improv}

In this section we describe the method used in \cite{MR3177759} to prove \eqref{improved-sub} and similar improvements of the \eqref{GNS} inequalities in the subcritical case $p\in(2,2^*)$. Inequalities \eqref{GNS} and their limiting case corresponding to $p=2$ can be written as:
\[
\i\ge d\,\e\quad\mbox{where}\quad\i:=\nrmS{\nabla u}2^2\quad\mbox{and}\quad\e:=\left\{\begin{array}{l}
\frac1{p-2}\(\nrmS up^2-\nrmS u2^2\)\quad\mbox{if}\quad p\neq2\,,\\[4pt]
\frac12\,\isd{|u|^2\,\ln\Big(\frac{|u|^2}{\nrmS u2^2}\Big)}\quad\mbox{if}\quad p=2\,,\\
\end{array}\right.
\]
for any $u\in \mathrm H^1(\mathbb S^d)$. By homogeneity we can assume that $\nrmS u2=1$. Let $I_p:=[0,+\infty)$ if $p\ge2$ and $I_p:=[0,1/(2-p))$ if $p\in(1,2)$.
\begin{theorem} [{\cite[Theorem~1.1]{MR3177759}}]\label{Thm:Main} Let $d\ge3$ and $p\in(2,2^*)$. With the above notation and conventions, there is an explicit function $\varphi$, such that $\varphi(0)=0$, $\varphi'(0)=1$, and $\varphi''>0$ on $I_p$, for which
\be{Phi}
\i\ge d\,\varphi(\e)\,.
\ee
\end{theorem}
\noindent Similar results can be proved in dimension $d=1$ and $d=2$: see~\cite{MR2381156,MR3177759,Dolbeault_2020,Brigati_2023}. Since
\[
\i-d\,\e\ge d\,\big(\varphi(\e)-\e\big)\ge0\,,
\]
it is clear that the equality case in \eqref{improved-sub} holds if and only if $\e=0$, that is, if $u$ is a constant. Moreover $\varphi(\e)-\e$ measures a distance to the constants, for instance in $\mathrm L^1(\mathbb S^d)$ norm using a generalized Cisz\'ar-Kullback-Pinsker inequality, and $\varphi(s)-s\sim\varphi''(0)\,s^2/2$ as $s\to0$. To prove \eqref{improved-sub} in the stronger homogeneous Sobolev norm $\dot{\mathrm H}^1(\R^d)$, it is enough to define the convex function
\[
\Psi_{d,p}(s):=s-d\,\varphi^{-1}\(\frac sd\)\,.
\]
It is elementary to verify that $\Psi_{d,p}(0)=\Psi_{d,p}'(0)=0$ and that $\Psi_{d,p}(s)>0$ if $s\neq 0$.

\medskip Let us give a sketch of the proof of Theorem~\ref{Thm:Main} based on the method of \cite{MR3177759}. As a first step, using Schwarz foliated symmetrization (see for instance \cite[Section~2]{MR3177759} for references) and cylindrical coordinates on $\mathbb S^d\subset\R^{d+1}$, we can reduce the problem to functions depending only on one coordinate $z\in[-1,1]$ corresponding to the South Pole -- North Pole axis. In other words, we consider a function $u(\omega)=f(z)$, where $\omega=(\omega_1,\omega_2,\ldots,\omega_{d+1})$ and $z=\omega_{d+1}$, and find that
\[
\nrmS{\nabla u}2^2=\int_{-1}^1 |f'(z)|^2\,\big(1-z^2\big) \, d\sigma_d(z) \quad\mbox{and}\quad\nrmS uq^q=\int_{-1}^1 |f(z)|^q\, d\sigma_d(z) \,,
\]
where $'$ denotes the $z$-derivative and where
\[
d\sigma_d(z):=\frac{\Gamma(\tfrac{d+1}2)}{\sqrt\pi\,\Gamma(\tfrac d2)}\,\big(1-z^2\big)^{\frac d2-1}\,dz\,.
\]
In particular, the Laplace-Beltrami operator on $\mathbb S^d$ is reduced to the {\it ultraspherical} operator: $\Delta u=\mathcal L\,f:=\big(1-z^2\big)\,f''-d\,z\,f'$ on $(-1,1)$.

The key idea is to prove that $\i-d\,\varphi(\e)$ is monotone nonincreasing under the action of
\be{NLDE}
\frac{\partial\rho}{\partial t}=\mathcal L\,\rho^m \,,
\ee
where $m=1$ corresponds to the heat flow, $m>1$ to the porous medium flow and $m<1$ to the fast diffusion flow. Here we choose $\rho=|f|^p$ so that $\nrm up$ is conserved under the action of \eqref{NLDE}. It is convenient to introduce the exponent $\beta$ such that
\[
m=1+\tfrac2p\,\big(\tfrac1\beta-1\big)
\]
and consider the function $f=w^\beta$, such that $w^{\beta\,p}=\rho$, which solves
\be{Eqn:Nonlinear}
\frac{\partial w}{\partial t}=w^{2-2\beta}\(\mathcal L w+\kappa\,\frac{|w'|^2}{w}\)\quad\mbox{with}\quad\kappa:=\beta\,(p-2)+1\,.
\ee
The \emph{carr\'e du champ} method shows that $\frac d{dt}\(\i-d\,\e\)\le0$ if the function
\[
\gamma(\beta):=-\,\big(\tfrac{d-1}{d+2}\,(\kappa+\beta-1)\big)^2+\kappa\,(\beta-1)+\,\tfrac d{d+2}\,(\kappa+\beta-1)
\]
takes nonnegative values, which amounts to $m_-(d,p)\le m\le m_+(d,p)$ where
\[
m_\pm(d,p):=\frac1{(d+2)\,p}\(d\,p+2\pm\sqrt{d\,(d-2)\,(p-1)\,(2^*-p)}\,\)\,.
\]
Using $\lim_{t\to+\infty}\(\i(t)-d\,\e(t)\)=0$, we conclude that $\i(t)-d\,\e(t)\ge0$ for any $t\ge0$ and, as a special case, for $t=0$: this proves \eqref{GNS} for an arbitrary initial datum. If $m_-(d,p)<m<m_+(d,p)$, we find that $\gamma(\beta)>0$, which leaves some space for an improvement. A more detailed and quite lengthy computation along the flow \eqref{Eqn:Nonlinear} yields
\[
\frac d{dt}\big(\mathsf i-\,d\,\varphi(\mathsf e)\big)\le\frac\gamma{\beta^2}\,\frac{\mathsf i-\,d\,\varphi(\mathsf e)}{\big(1-(p-2)\,\mathsf e\big)^\delta}\,\frac{d\e}{dt}
\]
where $\delta=1$ if $1\le p\le2$ and $\delta:=\frac{2-(4-p)\,\beta}{2\,\beta\,(p-2)}$ if $p>2$, if $\varphi$ solves
\[
\frac{d\varphi}{ds}=1+\frac\gamma{\beta^2}\,\frac{\varphi(s)}{\(1\,-\,(p-2)\,s\)^\delta}\,,\quad\varphi(0)=0\,.
\]
This proves \eqref{Phi} for any $t\ge0$ as a consequence of the monotonicity of $\i-d\,\varphi(\e)$ and the fact that $\lim_{t\to+\infty}\big(\i(t)-d\,\varphi(\e(t))\big)=0$. See~\cite[Appendix~B.4]{Brigati_2023} for detailed justifications.\qed

\medskip As a consequence of the \emph{carr\'e du champ} method, optimal constants in Gagliardo-Nirenberg-Sobolev inequalities determine optimal rates of decay in some associated fast diffusion equations. A stability result for a functional inequality can usually be rephrased as an improved inequality, under appropriate constraints. Measured by the dissipation of some entropy functional, this means a faster convergence rate towards an asymptotic profile for solutions of the associated evolution equation: see~\cite{MR3024094} for an application of stability for Gagliardo-Nirenberg-Sobolev inequalities to the proof of a quantitative rate of convergence for solutions of the Keller-Segel system, and~\cite{Brigati_2023} for explicit, improved, convergence rates based on~\eqref{GNS} in the subcritical regime, in the case of the sphere. On the Euclidean space, improved convergence rates based on stability for Gagliardo-Nirenberg-Sobolev inequalities (subcritical and critical cases) and entropy methods are studied in~\cite{BDNS}. Also see~\cite{MR3567822,brigati2024stability} for the limiting case of the logarithmic Sobolev inequality and corresponding improved decay rates for linear heat flows. Notice that stability estimates can also be applied in the perspective of estimating {\it a posteriori} errors both from theoretical (test functions) and numerical (approximate numerical solutions) points of view.

\newpage\section{Stability of the Sobolev inequality: proof of Theorem~\ref{main}}\label{Sec:stabilitySob}

In this section we explain the general ideas of the proof of Theorem~\ref{main} in \cite{DEFFL}.

\subsection{On the stability proof by Bianchi and Egnell}
The strategy of Bianchi-Egnell to prove~\eqref{eq:bianchi-egnell0} is based on two main steps:
\begin{enumerate}
\item A local stability estimate in a neighborhood of $\mathcal M$, obtained by a local spectral analysis.
\item A reduction of the global estimate to the local estimate by the concentration-compactness method based on Lions' analysis (see \cite{MR834360}).
\end{enumerate}
Theorem \ref{main} is a significant improvement of Bianchi-Egnell's result, as it contains a dimensionally sharp lower estimate for the best stability constant.
Our strategy is to make the \it second step constructive \rm and the \it first one explicit, \rm with much more detailed estimates.

\subsection{Strategy of the proof of Theorem \ref{main}}

The proof is divided into several steps:
\begin{enumerate}
\item Local analysis: prove the inequality for {\it nonnegative} functions close to $\mathcal M$ with an explicit remainder term. The analysis is quite involved: it relies on ``cuttings'' at various heights, the use of uniform bounds on spherical harmonics and some delicate concavity properties.
\item Local to global extension: prove the inequality for {\it nonnegative} functions far from $\mathcal M$ using the {\it competing symmetries} of \cite{CarlenLoss} and a continuous Steiner symmetrization.
\item Deduce the inequality for {\it sign-changing} functions from the inequality for {\it nonnegative} functions by a concavity argument.
\end{enumerate}

\medskip\noindent$\bullet$ In the first step, in order to obtain uniform estimates as $d\to+\infty$, we need to expand
\[
(1+r)^{2^*}-1-2^*\,r
\]
with an accurate remainder term, for all $r\ge-\,1$. To do that, we ``cut $r$ into pieces" by defining
\[
r_1:=\min\{r,\gamma\}\,,\quad r_2:=\min\big\{(r-\gamma)_+,M-\gamma\big\}\quad\mbox{and}\quad r_3:=(r-M)_+\,,
\]
where $\gamma$ and $M$ are suitable parameters satisfying $0<\gamma<M.$ Furthermore, define
\[\textstyle
\theta:=2^*-2=\frac4{d-2}\,.
\]
Notice that $\theta\in(0,1]$ if $d\ge6$ and $\lim_{d\to+\infty}\theta(d)=0$.
\begin{lemma}[{\cite[Proposition~2.9]{DEFFL}}] Given $d\ge6$, $r\in [-1,\infty)$, and $\overline M\in[\sqrt e,+\infty)$, we have
\begin{multline*}\label{eq:propinitial}
\hspace*{-12pt}(1+r)^{2^*} - 1 - 2^*\,r\le \tfrac12\,2^*\,(2^*-1)\,(r_1+r_2)^2 + 2\,(r_1+r_2)\,r_3 + \(1+C_M\,\theta\,\overline M^{-1}\ln\overline M\) r_3^{2^*}\\
+ \(\tfrac32\,\gamma\,\theta\,r_1^2 + C_{M,\overline M}\,\theta\,r_2^2\) \1_{\{r\le M\}} + C_{M,\overline M}\,\theta\,M^2\,\1_{\{r>M\}}\,,
\end{multline*}
where all the constants in the above inequality are explicit.
\end{lemma}
One can then prove that there exist computable constants $\epsilon_1$, $\epsilon_2$, $k_0$, and $\epsilon_0\in(0,1/\theta)$, such that
\[
\nrmS{\nabla r}2^2+\bfA\,\nrmS r2^2-\bfA\,\nrmS{1+r}{2^*}^2\ge\frac{4\,\epsilon_0}{d-2}\(\nrmS{\nabla r}2^2+\bfA\,\nrmS r2^2\)+\sum_{k=1}^3 I_k\,,
\]
with $\bfA:= \tfrac14\,d\,(d-2)$ and
\[\begin{aligned}
I_1 &\textstyle := (1-\theta\,\epsilon_0) \isd{\(|\nabla r_1|^2 + \bfA\,r_1^2\)} - \bfA\,(2^*-1 + \epsilon_1\,\theta) \isd{r_1^2} + \bfA\,k_0\,\theta \isd{(r_2^2+r_3^2)} \,,\\
I_2 &\textstyle := (1-\theta\,\epsilon_0) \isd{\(|\nabla r_2|^2 + \bfA\,r_2^2\)} - \bfA\,\big(2^*-1 + (k_0 + C_{\epsilon_1,\epsilon_2})\,\theta\big) \isd{r_2^2} \,, \\
I_3 &\textstyle := (1-\theta\,\epsilon_0) \isd{\(|\nabla r_3|^2 + \bfA\,r_3^2\)} - \tfrac 2{2^*}\,\bfA\,(1+\epsilon_2\,\theta) \isd{r_3^{2^*}} - \bfA\,k_0\,\theta \isd{r_3^2}\,.
\end{aligned}\]
Next, one can use spectral gap estimates to prove $I_1\ge 0$ and the Sobolev inequality to prove
$I_3\ge 0$, noting that the extra coefficient $2/2^*<1$ gives enough room to accommodate all error terms.
Finally, using that $\mu\big(\{ r_2 >0\}\big)$ is small, an improved spectral gap inequality allows us to show
$I_2\ge 0$.

If $d=3$, $4$, or $5$, we can rely on a simpler Taylor expansion that can be found in~\cite[Proposition~2.7]{DEFFL}. As a consequence, the following result has been proved.
\begin{theorem}[{\cite[Theorem~2.1]{DEFFL}}]\label{thm:unifboundclose} There are explicit constants $\epsilon_0\in(0,1/3)$ and $\delta\in(0,1/2)$ such that for all $d\ge 3$ and for all nonnegative $u=1+r\in \mathrm H^1(\mathbb S^d)$ with
\[
\nrmS r{2^*}^2\le\frac\delta{1-\delta}\,,\quad\isd r = 0\quad\mbox{and}\quad\isd{\omega\,r}=0\,,
\]
one has
\[
\nrmS{\nabla u}2^2+\bfA\,\nrmS u2^2-\bfA\,\nrmS u{2^*}^2\ge \frac{4\,\epsilon_0}{d-2}\(\nrmS{\nabla r}2^2+\bfA\,\nrmS r2^2\)\,.
\]
\end{theorem}

Let us define the stability quotient
\[
\mathcal E(f) :=\frac{ \Vert \nabla f \Vert^2_{\mathrm L^2(\R^d)} - S_d\,\Vert f \Vert^2_{\mathrm L^{2^*}(\R^d)}} {\inf_{\bb\in\mathcal M} \Vert \nabla f - \nabla\bb \Vert^2_{\mathrm L^2(\R^d)}}
\]
and consider the infimum
\[
\Imu(\delta) := \inf\Big\{ \mathcal E(f) \,:\,0\le f\in\dot{\mathrm H}^1(\R^d)\setminus\mathcal M\,,\quad\inf_{\bb\in\mathcal M} \Vert \nabla f - \nabla\bb\Vert^2_{\mathrm L^2(\R^d)} \le \delta\,\Vert \nabla f \Vert^2_{\mathrm L^2(\R^d)} \Big\}\,.
\]
For a given $f$, up to a conformal transformation, we can assume that $\inf_{\bb\in\mathcal M} \Vert \nabla f - \nabla\bb\Vert^2_{\mathrm L^2(\R^d)}$ is realized by the Aubin-Talenti function $g=\bb_\star$ with
\be{gstar}
\bb_\star(x):=|\Sp^d|^{-\frac{d-2}{2\,d}} \left( \frac2{1+|x|^2}\right)^\frac{d-2}2\quad\forall\,x\in\R^d\,,
\ee
use the inverse stereographic projection to transform $f$ and $g$ respectively into $u=1+r$ and $1$, and notice that
\begin{align*}
&\bfA\,\nrmS r{2^*}^2\le\nrmS{\nabla r}2^2+\bfA\,\nrmS r2^2=\inf_{\bb\in\mathcal M} \Vert \nabla f - \nabla\bb\Vert^2_{\mathrm L^2(\R^d)}\,,\\
&\Vert \nabla f \Vert^2_{\mathrm L^2(\R^d)}=\nrmS{\nabla u}2^2+\bfA\,\nrmS u2^2=\nrmS{\nabla r}2^2+\bfA\,\nrmS r2^2+\bfA\,,
\end{align*}
where the first line follows from~\eqref{Sob-sphere}. We deduce from the condition $\inf_{\bb\in\mathcal M} \Vert \nabla f - \nabla\bb\Vert^2_{\mathrm L^2(\R^d)} \le \delta\,\Vert \nabla f \Vert^2_{\mathrm L^2(\R^d)}$ that
\[
\bfA\,\nrmS r{2^*}^2\le\nrmS{\nabla r}2^2+\bfA\,\nrmS r2^2\le\frac{\delta\,\bfA}{1-\delta}
\]
and apply Theorem \ref{thm:unifboundclose} to obtain
\[
\mathcal E(f) = \frac{\nrmS{\nabla u}2^2+\bfA\,\nrmS u2^2-\bfA\,\nrmS u{2^*}^2}{\nrmS{\nabla r}2^2+\bfA\,\nrmS r2^2}\ge\frac{4\,\epsilon_0}{d-2}\,.
\]
This completes the \it local \rm analysis where, by homogeneity, the scale is fixed in terms of $\Vert \nabla f \Vert_{\mathrm L^2(\R^d)}^2$. With the notation of Theorem \ref{thm:unifboundclose}, 
\[
\Imu(\delta)\ge\frac{4\,\epsilon_0}{d-2}\,.
\]

\medskip\noindent$\bullet$ In Step 2, we deal with nonnegative functions $f$ that are not close to the manifold $\mathcal M$, \it i.e.\rm, such~that
\be{farway}
\inf_{\bb\in\mathcal M} \Vert \nabla f-\nabla\bb \Vert_{\mathrm L^2(\R^d)}^2 > \delta\,\Vert \nabla f \Vert_{\mathrm L^2(\R^d)}^2\,.
\ee
The first ingredient is the method of \it competing symmetries \rm \cite{CarlenLoss} of Carlen and Loss. Consider any nonnegative function $f\in\dot{\mathrm H}^1(\R^d)$ and let
\[
(Uf)(x) := \(\frac2{|x-e_d|^2}\right)^{\frac{d-2}2} f\left(\frac{x_1}{|x-e_d|^2}, \ldots, \frac{x_{d-1}}{|x-e_d|^2}, \frac{|x|^2-1}{|x-e_d|^2}\)\quad\mbox{where}\quad e_d=(0,\ldots, 0, 1)\in\R^d\,,
\]
and notice that $\mathcal E(Uf) = \mathcal E(f)$. We also consider the symmetric decreasing rearrangement $\mathcal R f=f^*$, with the properties that $f$ and $f^*$ are equimeasurable, and that $\Vert \nabla f^* \Vert_{\mathrm L^2(\R^d)} \le \Vert \nabla f\Vert_{\mathrm L^2(\R^d)}$. The following result is taken from the proof of \cite[Theorem~3.3]{CarlenLoss}.
\begin{theorem}[{\cite{CarlenLoss}}]\label{competingsymmetries} Let $f \in \mathrm L^{2^*}(\R^d)$ be a nonnegative function with \hbox{$\Vert f \Vert_{\mathrm L^{2^*}(\R^d)}=1$}. The sequence $f_n = (\mathcal RU)^n f$ is such that $\lim_{n\to+\infty} \Vert f_n - \bb_*\Vert_{\mathrm L^{2^*}(\R^d)} = 0$. If $f\in\Hdot$, then $(\Vert \nabla f_n \Vert_{\mathrm L^2(\R^d)})_{n\in\N}$ is a nonincreasing sequence. \end{theorem}
Whether $f_n$ satisfies \eqref{farway} for all $n\in\N$ or not, we face an alternative.
\begin{lemma}[{\cite[Lemma~3.5]{DEFFL}}]\label{alternatives} Let $f$ be as in Theorem~\ref{competingsymmetries} and such that \eqref{farway} holds and let $f_n = (\mathcal RU)^n f$. Then either $\inf_{\bb\in\mathcal M} \Vert \nabla f_n - \nabla\bb\Vert_{\mathrm L^2(\R^d)}^2\ge\delta\,\Vert \nabla f_n \Vert_{\mathrm L^2(\R^d)}^2$ for all $n\in\N$, or there exists $ n_0\in\N$ such that
\[
\inf_{\bb\in\mathcal M} \Vert \nabla f_{n_0} - \nabla\bb\Vert_{\mathrm L^2(\R^d)}^2 > \delta\,\Vert \nabla f_{n_0} \Vert_{\mathrm L^2(\R^d)}^2
\quad\mbox{and}\quad \inf_{\bb\in\mathcal M} \Vert \nabla f_{n_0+1} - \nabla\bb\Vert_{\mathrm L^2(\R^d)}^2 < \delta\,\Vert \nabla f_{n_0+1} \Vert_{\mathrm L^2(\R^d)}^2\,.
\]
\end{lemma}
In the first case, that is if $f_n$ satisfies \eqref{farway} for all $n\in\N$, we have
\[
\lim_{n\to+\infty}\Vert \nabla f _n\Vert^2_2\leq\frac1\delta\,\lim_{n\to+\infty}\,\inf_{\bb\in\mathcal M} \Vert \nabla f_n - \nabla\bb\Vert_2^2= \frac1\delta\left(\lim_{n\to+\infty} \Vert \nabla f_n \Vert^2_2 - S_d\,\Vert f \Vert_{2^*} ^2\right)
\]
where the last equality arises as a consequence of the properties of $(f_n)_{n\in\N}$ (see \cite[Lemma~3.4]{DEFFL}). Combined with the simple estimate
\[
\mathcal E(f) = \frac{\Vert \nabla f \Vert^2_2 - S_d\,\Vert f \Vert_{2^*}^2} {\inf_{\bb\in\mathcal M} \Vert \nabla f - \nabla\bb\Vert^2_2 } \ge \frac{\Vert \nabla f \Vert^2_2 - S_d\,\Vert f \Vert_{2^*}^2} {\Vert \nabla f \Vert^2_2 } \geq \frac{\Vert \nabla f_n \Vert^2_2 - S_d\,\Vert f \Vert_{2^*}^2} {\Vert \nabla f _n\Vert^2_2}\,,
\]
we can take the limit as $n\to+\infty$ and obtain $\mathcal E(f)\ge\delta$. In the second case, we adapt a strategy that has some similarities to one used by Christ in \cite{Christ} and is based on a continuous rearrangement {\it flow} $(\mathsf f_\tau)_{n_0\le \tau < n_0+1}$ with $\mathsf f_{n_0}=Uf_{n_0}$ such~that
\[
\Vert \mathsf f_\tau \Vert_{\mathrm L^{2^*}(\R^d)}= \Vert \mathsf f_{n_0} \Vert_{\mathrm L^{2^*}(\R^d)}\,,\;\tau \mapsto \Vert \nabla \mathsf  f_\tau \Vert_{\mathrm L^2(\R^d)}\;\mbox{\it is nonincreasing, and}\;\lim_{\tau \to n_0+1}  \mathsf  f_\tau=f_{n_0+1}\,.
\]
Choosing the smallest $\tau_0\in(n_0,n_0+1)$ such that $\inf_{\bb\in\mathcal M} \Vert \nabla(\mathsf f_{\tau_0}-\bb) \Vert_{\mathrm L^2(\R^d)}^2 = \delta\,\Vert \nabla \mathsf f_{\tau_0} \Vert_{\mathrm L^2(\R^d)}^2$ and using
$\Vert f\Vert_{\mathrm L^{2^*}(\R^d)}=\Vert \mathsf f_{n_0} \Vert_{\mathrm L^{2^*}(\R^d)}=\Vert\mathsf f_{\tau_0} \Vert_{\mathrm L^{2^*}(\R^d)}$, this gives:
\[
1-S_d\,\frac{\Vert f \Vert^2_{\mathrm L^{2^*}(\R^d)}} {\Vert \nabla f \Vert^2_{\mathrm L^2(\R^d)}}
\ge1-S_d\,\frac{\Vert\mathsf f_{\tau_0} \Vert^2_{\mathrm L^{2^*}(\R^d)}} {\Vert \nabla\mathsf f_{\tau_0} \Vert^2_{\mathrm L^2(\R^d)}}=\delta \, \mathcal E(\mathsf f_{\tau_0})\ge\delta\,\Imu(\delta)\,.
\]
As a consequence, we obtain $C_{d,\rm BE}^{\rm pos}:=\inf_{f\ge0}\mathcal E(f)\ge\min\left\{\delta,\delta\,\Imu(\delta)\right\}=\delta\,\Imu(\delta)$ because $\Imu(\delta)<1$. To build the flow, we refer to \cite{MR1330619,MR1758811} and to \cite{DEFFL} for further references. The existence of $\tau_0$ requires a discussion that can be found in~\cite[Section~3.1.2]{DEFFL}.

\medskip\noindent$\bullet$ The third step is to remove the positivity assumption of Theorem \ref{main} as in \cite[Section~3.2]{DEFFL}. Take $f=f_+-f_-$ with $\nrm{f}{2^*}=1$ and define $m :=\nrm{f_-}{2^*}^{2^*}$. Without loss of generality one may assume that $1-m=\nrm{f_+}{2^*}^{2^*}>1/2$. The positive concave function
\[
h_d(m):=m^\frac{d-2}d+(1-m)^\frac{d-2}d-1
\]
satisfies
\[
2\,h_d(1/2)\,m\le h_d(m)\,,\quad h_d(1/2)= 2^{2/d}-1 \,.
\]
With $\Deficit (f)=\nrm{\nabla f}2^2-S_d \nrm f{2^*}^2$, one finds $g_+\in \mathcal M$ such that
\[
\Deficit (f)\ge C_{d,\rm BE}^{\rm pos} \nrm{\nabla f_+-\nabla g_+}2^2 +\frac{2\,h_d(1/2)}{h_d(1/2)+1} \nrm{\nabla f_-}2^2\,,
\]
and therefore
\[
C_{d,\rm BE}\ge\tfrac12\,\min\left\{\delta\,\Imu(\delta), \frac{2\,h_d(1/2)}{h_d(1/2)+1}\right\}\,.
\]

\noindent$\bullet$ Combining all estimates of the three previous steps completes the proof of Theorem \ref{main}.\qed

\section{Stability of the Gaussian logarithmic Sobolev inequality: proof of Theorem~\ref{logsob}}\label{Sec:stabilitylogSob}

In this section we describe the main steps in the proof of the stability estimate for the Gaussian logarithmic Sobolev inequality \eqref{Ineq:logSob} by considering the large dimensional limit of~\eqref{stabgrandd}, as it appears in \cite[Section~4]{DEFFL}. With $u=f/g_\star$, where $g_\star$ is given by~\eqref{gstar}, Inequality \eqref{stabgrandd} can be rewritten as
\begin{multline*}
\ird{|\nabla u|^2\, g_\star^2\nc}+d\,(d-2)\,\ird{|u|^2\, g_\star^{2^*}\nc}-d\,(d-2)\,\nrm{g_\star}{2^*}^{2^*-2}\nc\(\ird{|u|^{2^*}\, g_\star^{2^*}\nc}\)^{2/2^*}\hspace*{2cm}\\
\ge\frac\beta d\(\ird{|\nabla u|^2\, g_\star^2\nc}+d\,(d-2)\,\ird{|u-g_d/g_\star|^2\,g_\star^{2^*}}\),
\end{multline*}
where $g_d\in\mathcal M$ realizes $\inf_{\bb\in\mathcal M} \Vert \nabla f - \nabla\bb \Vert^2_{\mathrm L^2(\R^d)}$. For some parameters $a_d$, $b_d$ and $c_d$, we can write that $g_d(x)= c_d\,(a_d+|x-b_d|^2)^{1-d/2}$. We rescale the function $u$ according to
\[
u(x)=v\(r_d\,x\)\quad\forall\,x\in\R^d\,,\quad r_d=\sqrt{\tfrac d{2\,\pi}}
\]
and consider the function $w_v^d$ such that $w_v^d\(r_d\,x\)=g_d(x)/g_\star(x)$. Hence
\begin{multline*}
\irdmu{|\nabla v|^2\(1+\tfrac1{r_d^2}\,|x|^2\)^2}
\ge2\,\pi\,(d-2)\(\(\irdmu{|v|^{2^*}}\)^{2/2^*}-\irdmu{|v|^2}\)\hspace*{2cm}\\
+\frac\beta d\(\irdmu{|\nabla v|^2}+2\,\pi(d-2)\nc\,\irdmu{|v-w_v^d|^2}\)
\end{multline*}
where $d\mmu_d=Z_d^{-1}\,g_\star^{2^*}\,dx$ is the probability measure given by
\[
d\mmu_d(x):=\frac1{Z_d}\(1+\tfrac1{r_d^2}\,|x|^2\)^{-d}\,dx\quad\mbox{with}\quad Z_d=\frac{2^{1-d}\,\sqrt\pi}{\Gamma\(\tfrac{d+1}2\)}\(\frac d2\)^\frac d2 \,.
\]
Our goal is to take the limit $d\to +\infty$ when one considers functions $v(x)$ depending only on $y\in\R^N$, with $x=(y,z)\in\R^N\times\R^{d-N}\approx\R^d$, for some fixed integer $N$. With $|x|^2=|y|^2+|z|^2$, notice that
\[
1+\tfrac1{r_d^2}\,|x|^2=1+\tfrac1{r_d^2}\,\big(|y|^2+|z|^2\big)=\(1+\tfrac1{r_d^2}\,|y|^2\)\(1+\tfrac{|z|^2}{r_d^2+|y|^2}\)
\]
and, as a consequence,
\begin{align*}
&\textstyle\lim_{d\to+\infty}\(1+\tfrac1{r_d^2}\,|y|^2\)^{-\frac{N+d}2}=e^{-\pi\,|y|^2}\,,\\
&\textstyle\lim_{d\to+\infty}\irdmu{|v(y)|^2}=\irNg{|v|^2}\,,\\
&\textstyle\lim_{d\to+\infty}\irdmu{|\nabla v|^2\(1+\tfrac1{r_d^2}\,|x|^2\)^2}=4\,\irNg{|\nabla v|^2}\,,
\end{align*}
where $d\gamma(y)=e^{-\pi\,|y|^2}\,dy$. However, the function $w_v^d$ depends on $d$ and the main difficulty is to obtain enough estimates on the parameters $a_d$, $b_d$ and $c_d$ to pass to the limit after integrating in all integrals with respect to $z$, which completes the proof of~\eqref{LSI:stab}. See \cite[Section~4]{DEFFL} for further details. \qed

\section{Stability of the logarithmic Sobolev inequality: a new proof of Theorem~\ref{logsob}}\label{sec:logsob}

As explained in the Introduction and in Section~\ref{Sec:stabilitylogSob}, in \cite{DEFFL} we proved Theorem~\ref{logsob} as a corollary of Theorem~\ref{main}. In this section we give a new and complete proof of  Theorem~\ref{logsob} that is independent of  Theorem~\ref{main}.

This gives us the opportunity to present some of the details both of the infinite dimensional limit mentioned in Section \ref{Sec:stabilitylogSob} and of the combination of continuous symmetrization and discrete competing symmetries flows mentioned in Section \ref{Sec:stabilitySob}. In both cases the arguments in the present section are slightly simpler, but illustrate well the main underlying ideas. We also emphasize that, while it may have seemed that the latter argument relies on the conformal invariance of the Sobolev inequality (which enters through the operator $U$ in the competing symmetries argument), this is not the case: in the setting of log-Sobolev~\eqref{LSI}, where there is no such conformal symmetry, there is another operator that serves the same purpose. 

\smallskip
Just like in Section \ref{Sec:stabilitySob}, we prove the quantitative version of the sharp logarithmic Sobolev inequality (Theorem \ref{logsob}) in two steps, one close and one far from the set of optimizers. Let us start with the following result that replaces Theorem~\ref{thm:unifboundclose} here. In fact, it is a consequence of Theorem~\ref{thm:unifboundclose}.
\begin{theorem}\label{unifboundcloselogsob}
There are explicit constants $\eta>0$ and $\delta\in(0,1/2)$ such that for all $N\in\N$ and for all for all nonnegative $u=1+r\in \mathrm H^1(\mathbb S^d)$ satisfying
\begin{equation}
\label{eq:smalllogsob}
\int_{\R^N} r^2\,d\gamma \le \frac\delta{1-\delta}
\end{equation}
and
\begin{equation}
\label{eq:orthologsob}
\int_{\R^N} r\,d\gamma = 0 = \int_{\R^N} x_j\,r\,d\gamma\,,
\quad j=1,\,2,\ldots,N\,,
\end{equation}
one has
\[
\int_{\R^N} |\nabla u|^2\,d\gamma - \pi \int_{\R^N} |u|^2 \ln\(\frac{|u|^2}{\| u \|_{\mathrm L^2(\gamma)}^2}\)\,d\gamma\ge \eta \int_{\R^N} r^2\,d\gamma\,.
\]
\end{theorem}
\noindent The constant $\delta$ coincides with the corresponding constant in Theorem~\ref{thm:unifboundclose} and $\eta = 2\,\pi\,\epsilon_0$.\begin{proof} Notice that $x\in \mathrm L^2(\gamma)$, so the orthogonality constraints raise no integration issues.
We denote $\Sigma_d:= \{ x\in\R^{d+1} :\,|x| = \rho_d\}$ with $\rho_d:=\sqrt{d/(2\,\pi)}$. The factor of $1/(2\,\pi)$ in the definition of $\rho_d$ is necessary to get the $\pi$ in the exponent of the Gaussian density. We integrate on $\Sigma_d$ with respect to the uniform probability measure $d\mu_d$. Here, with a straightforward abuse, we use the same notation for the uniform probability measure on $\Sigma_d$ as on the unit sphere. By rescaling our result in Theorem~\ref{thm:unifboundclose} we find that
\begin{multline}
\label{eq:ineqresc}
\int_{\Sigma_d} |\nabla R|^2\,d\mu_d - \pi \,\frac{d-2}2 \(\(\int_{\Sigma_d} (1+R)^\frac{2\,d}{d-2}\,d\mu_d\)^\frac{d-2}d - \int_{\Sigma_d} (1+R)^2\,d\mu_d\)\\
\ge 2\,\pi\,\epsilon_0 \int_{\Sigma_d} \(\frac1\pi\,\frac 2{d-2}\,|\nabla R|^2 + R^2\)\,d\mu_d\,.
\end{multline}
This inequality is valid for all $R\in \mathrm H^1(\Sigma_d)$ such that
\begin{equation}
\label{eq:smallresc}
\(\int_{\Sigma_d} R^\frac{2\,d}{d-2}\,d\mu_d\)^\frac{d-2}d \le \frac\delta{1-\delta}
\end{equation}
and
\begin{equation}
\label{eq:orthoresc}
\int_{\Sigma_d} R\,d\mu_d = 0 = \int_{\Sigma_d} x_j\,R\,d\mu_d \,,
\quad j=1,\ldots, d+1\,.
\end{equation}
Given a function $r\in \mathrm H^1(\gamma)$ and $d>N$, we apply this inequality to the function
\[
R_d(x) := r(x_1,\ldots,x_N) - \int_{\Sigma_d} r\,d\mu_d - 2\,\pi\,\frac{d+1}d \sum_{n=1}^N x_n \int_{\Sigma_d} y_n\,r(y_1,\ldots,y_N)\,d\mu_d(y)
\]
for $x\in\Sigma_d$. This function satisfies the orthogonality conditions~\eqref{eq:orthoresc}. Note here that the functions $\sqrt{2\,\pi}\,\sqrt{(d+1)/d}\,x_j$ are $\mathrm L^2$-normalized on $\Sigma_d$.

We now use the well-known fact that, as $d\to+\infty$, the marginal of $d\mu_d$ corresponding to the first~$N$ coordinates converges to $d\gamma$. Thus,
\begin{align*}
& \lim_{d\to+\infty}\int_{\Sigma_d} |\nabla r|^2\,d\mu_d = \int_{\R^N} |\nabla r|^2\,d\gamma\,,
& & \lim_{d\to+\infty}\int_{\Sigma_d} r^2\,d\mu_d = \int_{\R^N} r^2\,d\gamma\,,\\
& \lim_{d\to+\infty}\int_{\Sigma_d} r\,d\mu_d = \int_{\R^N} r\,d\gamma = 0\,,
& & \lim_{d\to+\infty}\int_{\Sigma_d} y_n\,r(y_1,\ldots,y_N)\,d\mu_d(y) = \int_{\R^N} y_n\,r\,d\gamma = 0\,.
\end{align*}
From this we conclude easily that
\[
\lim_{d\to+\infty}\int_{\Sigma_d} |\nabla R_d|^2\,d\mu_d = \int_{\R^N} |\nabla r|^2\,d\gamma\,,
\quad
\lim_{d\to+\infty}\int_{\Sigma_d} R_d^2\,d\mu_d = \int_{\R^N} r^2\,d\gamma\,.
\]
With some modest amount of effort one also finds that
\[
\lim_{d\to+\infty}\int_{\Sigma_d} R_d^\frac{2\,d}{d-2}\,d\mu_d = \int_{\R^N} r^2\,d\gamma\,.
\]
Assuming that the inequality in~\eqref{eq:smalllogsob} is strict, the same is true for the left side when $d$ is sufficiently large, and consequently the smallness condition~\eqref{eq:smallresc} holds when $d$ is sufficiently large. Thus, inequality~\eqref{eq:ineqresc} is valid for all sufficiently large $d$. The equality case in~\eqref{eq:smalllogsob} can be obtained at the very end by a simple approximation argument. 

Now, we drop the gradient term in the right side and letting $d\to+\infty$ we infer that
\begin{equation*}
\int_{\R^N} |\nabla r|^2\,d\gamma - \pi\,\limsup_{d\to+\infty} \frac{d-2}2 \(\(\int_{\Sigma_d} (1+R_d)^\frac{2\,d}{d-2}\,d\mu_d\)^\frac{d-2}d \kern-12pt- \int_{\Sigma_d} (1+R_d)^2\,d\mu_d\)\ge 2\,\pi\,\epsilon_0 \int_{\R^N} r^2\,d\gamma\,.
\end{equation*}
Finally, we verify that
\begin{multline*}
\limsup_{d\to+\infty} \frac{d-2}2 \(\(\int_{\Sigma_d} (1+R_d)^\frac{2\,d}{d-2}\,d\mu_d\)^\frac{d-2}d \kern-12pt- \int_{\Sigma_d} (1+R_d)^2\,d\mu_d\)\\ = \int_{\R^N} (1+r)^2\, \ln\(\frac{(1+r)^2}{\| 1+r \|_{\mathrm L^2(\gamma)}^2}\)\,d\gamma\,.
\end{multline*}
In fact, if the orthogonality conditions were not present and the marginals would already be equal to their limit, this would follow from the fact that
\[
\lim_{p\to 1^+} \frac1{p-1} \(\(\int_{\R^N} h^p\,d\gamma\)^{1/p} - \int_{\R^N} h\,d\gamma\) = \int_{\R^N} h \ln\(\frac{h}{\int_{\R^N} h\,d\gamma}\)d\gamma\,,
\]
valid on any measure space for any nonnegative function $h$ that satisfies $h\in \mathrm L^1\cap \mathrm L^{p_0}(\gamma)$ for some \hbox{$p_0>1$}. Proving the latter fact is simple, as well as including the effect of the orthogonality conditions and the convergence of the marginals, so we shall omit it. These remarks complete the proof Theorem~\ref{unifboundcloselogsob}.
\end{proof}

We emphasize that in the previous proof we did not use Theorem~\ref{main}, but rather Theorem~\ref{thm:unifboundclose}. In this way we avoid having to control the distance to the set of optimizers in the high-dimensional limit, which seems harder than verifying the orthogonality conditions.

\begin{proof}[Proof of Theorem~\ref{logsob}]
As in the proof of Theorem~\ref{main}, we first prove the result for nonnegative functions and then extend it to sign changing solutions. Let us denote by $\kappa^{\rm pos}$ the stability constant in the stability inequality restricted to nonnegative functions.

\medskip\noindent\emph{Step 1.} Let $\eta$ and $\delta$ be as in Theorem~\ref{unifboundcloselogsob}. For $0\le u\in \mathrm H^1(\gamma)$ we distinguish two cases.

\smallskip\noindent$\bullet$ The first case is where
\[
\inf_{a\in\R^N\!,\,c\in\R} \int_{\R^N} \big|u - c\,e^{a\cdot x}\big|^2\,d\gamma \le \delta \int_{\R^N} |u|^2\,d\gamma\,.
\]
The infimum on the left-hand side is attained at some $a\in\R^N$ and $c\in\R$, as can be checked by optimizing $\int_{\R^N}\big|v-c\,e^{|a|^2/(2\,\pi)-\pi\,|x-a/\pi|^2/2}\big|^2\,dx$ where $v(x):=u(x)\,e^{-\,\pi\,|x|^2/2}$. Let
\[
\tilde u(y) := e^{-\,\,y\cdot a - \frac{|a|^2}{2\,\pi}}\,u\big(y+\tfrac a\pi\big)\,.
\]
Then, by a simple computation involving an integration by parts and a change of variables,
\[
\int_{\R^N} |\nabla\tilde u|^2\,d\gamma - \pi \int_{\R^N} |\tilde u|^2 \ln\(\frac{|\tilde u|^2}{\|\tilde u\|_{\mathrm L^2(\gamma)}^2}\)\,d\gamma =
\int_{\R^N} |\nabla u|^2\,d\gamma - \pi \int_{\R^N} |u|^2 \ln\(\frac{|u|^2}{\|u\|_{\mathrm L^2(\gamma)}^2}\)\,d\gamma\,.
\]
Therefore, the deficit of $\tilde u$ coincides with that of $u$, while the infimum for $\tilde u$ among all functions of the form \eqref{eq:optim-logSob} is attained at the constant $c\,\exp\big(|a|^2/(2\,\pi)\big)$. Finally, by multiplying $\tilde u$ with a constant, we may assume that this constant is equal to one. To summarize, we may assume without loss of generality that the infimum in the theorem is attained at $a=0$ and $c=1$.

Let us set $r:= u-1$. Then the minimality implies that $r$ satisfies the orthogonality conditions~\eqref{eq:orthologsob}. Moreover, we have
\[
\int_{\R^N} r^2\,d\gamma \le \delta \int_{\R^N} |u|^2\,d\gamma = \delta \(1 + \int_{\R^N} r^2\,d\gamma\),
\]
so the smallness condition~\eqref{eq:smalllogsob} is satisfied and we can apply Theorem~\ref{unifboundcloselogsob}. This yields the inequality in the theorem with a stability constant $\eta$.

\smallskip\noindent$\bullet$ Next, we consider the case where
\[
\inf_{a\in\R^N\!,\,c\in\R} \int_{\R^N} (u - c\,e^{x\cdot a})^2\,d\gamma > \delta \int_{\R^N} |u|^2\,d\gamma\,.
\]
We argue similarly as we did in Section~\ref{Sec:stabilitySob} concerning the Sobolev inequality, but there are some differences in this case.

For $f\in \mathrm L^2(\gamma)$ we denote by $Uf$ its Gaussian rearrangement, that is, the function on $\R^N$ whose superlevel sets have the form $\{x\in\R^N:\,x_1 < \mu\}$ for some $\mu\in\R$ and have the same $\gamma$-measure as the corresponding superlevel sets of $f$. Moreover, we denote
\[
Vf := e^{\frac\pi2\,|x|^2}\,\mathcal R \( e^{-\,\frac\pi2\,|x|^2}\,f\)\,,
\]
where $\mathcal R$ is, as before, the Euclidean rearrangement. Then, as shown in~\cite[Theorem 4.1]{CarlenLoss}, for any $0\le f\in \mathrm L^2(\gamma)$ one has
\[
f_n := (VU)^n f \to \|f\|_{\mathrm L^2(\gamma)}\quad\text{in}\quad\mathrm L^2(\gamma)\,.
\]
Moreover, $\|f_n\|_{\mathrm L^2(\gamma)} = \|f\|_{\mathrm L^2(\gamma)}$ and
\[
n\mapsto \int_{\R^N} |\nabla f_n|^2\,d\gamma - \pi \int_{\R^N} f_n^2 \ln\(\frac{f_n^2}{\| f_n \|_{\mathrm L^2(\gamma)}^2}\)\,d\gamma
\]
is nonincreasing. This is the analogue of Theorem \ref{competingsymmetries} in the present case.

We apply this procedure to our function $u$, which we assume here to be nonnegative, and obtain a sequence of functions $u_n$ with constant $\mathrm L^2(\gamma)$-norm. Moreover, since
\[
\inf_{a,\,c} \| u_n - c\,e^{a\cdot x} \|_{\mathrm L^2(\gamma)} \le \big\| u_n - \|u\|_{\mathrm L^2(\gamma)} \big\|_{\mathrm L^2(\gamma)} \to 0\quad\mbox{as}\quad n\to+\infty\,,
\]
there is an $n_0\in\N$ such that
\[
\inf_{a,\,c} \| u_{n_0} - c\,e^{a\cdot x} \|_{\mathrm L^2(\gamma)}^2\ge \delta\,\|u\|_{\mathrm L^2(\gamma)}^2 > \inf_{a,\,c} \| u_{n_0+1} - c\,e^{a\cdot x} \|_{\mathrm L^2(\gamma)}^2\,.
\]
This replaces Lemma \ref{alternatives}. We have
\begin{align*}
\frac{\int_{\R^N} |\nabla u|^2\,d\gamma - \pi \int_{\R^N} |u|^2 \ln\(\frac{|u|^2}{\| u \|_{\mathrm L^2(\gamma)}^2}\)\,d\gamma }{\inf_{a,\,c} \| u - c\,e^{a\cdot x} \|_{\mathrm L^2(\gamma)}^2}
&\ge \frac{\int_{\R^N} |\nabla u|^2\,d\gamma - \pi \int_{\R^N} |u|^2 \ln\(\frac{|u|^2}{\| u \|_{\mathrm L^2(\gamma)}^2}\)\,d\gamma }{\| u \|_{\mathrm L^2(\gamma)}^2}\\
&\ge \frac{\int_{\R^N} |\nabla u_{n_0}|^2\,d\gamma - \pi \int_{\R^N} u_{n_0}^2 \ln\(\frac{u_{n_0}^2}{\| u_{n_0} \|_{\mathrm L^2(\gamma)}^2}\)\,d\gamma }{\| u \|_{\mathrm L^2(\gamma)}^2}\,.
\end{align*}
We now use a continuous rearrangement flow to connect $u_{n_0}$ to $u_{n_0+1}$. More precisely, we a consider a family of functions $(\mathsf u_\tau)_{n_0\le\tau\le n_0+1}$, where $\mathsf u_{n_0}:=Uu_{n_0}$ and $\mathsf u_{n_0+1}:=u_{n_0+1}$. We define $\mathsf u_\tau$ as $e^{\frac{\pi}{2}\,|x|^2}$ times the continuous (Euclidean) rearrangement of $e^{-\frac{\pi}{2}\,|x|^2}\,U u_{n_0}$ at parameter $\tau$. In the same way as in \cite[Lemma 36]{DEFFL}, one sees that
\[
\tau \mapsto \inf_{a,\,c} \| \mathsf u_\tau - c\,e^{a\cdot x} \|_{\mathrm L^2(\gamma)}^2
\]
is continuous, and therefore there is a $\tau_0\in[n_0,n_0+1)$ such that
\[
\inf_{a,\,c} \| \mathsf u_{\tau_0} - c\,e^{a\cdot x} \|_{\mathrm L^2(\gamma)}^2 = \delta\, \|u\|_{\mathrm L^2(\gamma)}^2\,.
\]
It follows that
\begin{multline*}
\frac{\int_{\R^N} |\nabla u_{n_0}|^2\,d\gamma - \pi \int_{\R^N} u_{n_0}^2 \ln\(\frac{u_{n_0}^2}{\| u_{n_0} \|_{\mathrm L^2(\gamma)}^2}\)\,d\gamma }{\| u \|_{\mathrm L^2(\gamma)}^2}\\
\ge \frac{\int_{\R^N} |\nabla \mathsf u_{\tau_0}|^2\,d\gamma - \pi \int_{\R^N} \mathsf u_{\tau_0}^2 \ln\(\frac{\mathsf u_{\tau_0}^2}{\| \mathsf u_{\tau_0} \|_{\mathrm L^2(\gamma)}^2}\)\,d\gamma }{\| u \|_{\mathrm L^2(\gamma)}^2}\\
= \delta\,\frac{\int_{\R^N} |\nabla \mathsf u_{\tau_0}|^2\,d\gamma - \pi \int_{\R^N} \mathsf u_{n_0}^2 \ln\(\frac{\mathsf u_{\tau_0}^2}{\| \mathsf u_{\tau_0} \|_{\mathrm L^2(\gamma)}^2}\)\,d\gamma }{\inf_{a,\,c} \| \mathsf u_{\tau_0} - c\,e^{a\cdot x} \|_{\mathrm L^2(\gamma)}^2}\,.
\end{multline*}
We can apply the result in the first case to the function $\mathsf u_{\tau_0}$ and infer that the right side is larger or equal than $\kappa^{\rm pos}:=\delta\,\eta$. This concludes the proof in the case of nonnegative functions.

\medskip\noindent\emph{Step 2.} We now prove the theorem in the general case, that is, for sign-changing functions.

We shall use the notation
\[
D(v) := \int_{\R^N} |\nabla v|^2\,d\gamma - \pi \int_{\R^N} v^2 \ln\(\frac{v^2}{\|v\|_{\mathrm L^2(\gamma)}^2}\)\,d\gamma
\quad\text{for}\,v\in \mathrm H^1(\gamma)\,.
\]
Let $u=u_+-u_-\in \mathrm H^1(\gamma)$. By homogeneity we can assume $\|u\|_{\mathrm L^2(\gamma)}=1$. Replacing $u$ by $-\,u$ if necessary, we can also assume that
\[
m:= \|u_-\|_{\mathrm L^2(\gamma)}^2 \in [0,\tfrac12]\,.
\]
Then
\[
D(u) = D(u_+) + D(u_-) + \pi\,h(m)
\]
with
\[
h(p) := -\,\big( p\ln p + (1-p) \ln(1-p) \big)\,.
\]
Since the function $p\mapsto h(p)$ is monotone increasing and concave on $[0,1/2]$, it holds that
\[
h(p)\ge (2\ln 2)\,p\quad\forall\,p\in[0,\tfrac12]\,.
\]
Thus, with $\kappa^{\rm pos}$ denoting the constant from Step 1,
\begin{align*}
D(u) &\ge D(u_+) + (2\,\pi\, \ln 2)\, m\ge \kappa^{\rm pos} \inf_{a,\,c}\|u_+ - c\,e^{a\cdot x} \|_{\mathrm L^2(\gamma)}^2 + (2\,\pi\, \ln 2)\, \|u_-\|_{\mathrm L^2(\gamma)}^2\\
&\ge \frac12\, \min\big\{\kappa^{\rm pos},\, 2\,\pi\, \ln 2 \big\}\,\inf_{a,\,c} \| u - c\,e^{a\cdot x} \|_{\mathrm L^2(\gamma)}^2\,.
\end{align*}
This proves the inequality in the general case, with $\kappa = \frac12\, \min\big\{\kappa^{\rm pos},\, 2\,\pi\, \ln 2 \big\}$. It is straightforward to verify that $\kappa = \beta\,\pi/2$, $\beta$ being the constant in the statement of Theorem \ref{main}. This ends our second proof of Theorem \ref{logsob}.
\end{proof}

\bigskip{\small\noindent\Emph{Acknowledgements.}\/ Partial support through US National Science Foundation grants DMS-1954995 (R.L.F.) and DMS-2154340 (M.L.), as well as through the Deutsche Forschungsgemeinschaft (DFG, German Research Foundation) Germany's Excellence Strategy EXC-2111-390814868 (R.L.F.) and through TRR 352 – Project-ID 470903074 (R.L.F.), the French National Research Agency (ANR) project {\it Conviviality} (ANR-23-CE40-0003, J.D.) and the European Research Council under the Grant Agreement No. 721675 (RSPDE) {\it Regularity and Stability in Partial Differential Equations} (A.F.).\\ The authors are grateful to an anonymous referee for a careful rereading which helped to improve the manuscript.}\\
{\tiny\copyright~2024 by the authors. Reproduction of this article by any means permitted for noncommercial purposes. \hbox{\href{https://creativecommons.org/licenses/by/4.0/legalcode}{CC-BY 4.0}}}

\bigskip\raggedbottom

\begin{thebibliography}{10}

\bibitem{Aubin}
Thierry Aubin.
\newblock \'{E}quations diff\'{e}rentielles non lin\'{e}aires et probl\`eme de
  {Y}amabe concernant la courbure scalaire.
\newblock {\em J. Math. Pures Appl. (9)}, 55(3):269--296, 1976.

\bibitem{MR772092}
Dominique Bakry and Michel {\'E}mery.
\newblock Hypercontractivit\'e de semi-groupes de diffusion.
\newblock {\em C. R. Acad. Sci. Paris S\'er. I Math.}, 299(15):775--778, 1984.

\bibitem{Bakry-Emery85}
Dominique Bakry and Michel {\'E}mery.
\newblock Diffusions hypercontractives.
\newblock In {\em S\'eminaire de probabilit\'es, XIX, 1983/84}, volume 1123 of
  {\em Lecture Notes in Math.}, pages 177--206. Springer, Berlin, 1985.

\bibitem{MR808640}
Dominique Bakry and Michel {\'E}mery.
\newblock In\'egalit\'es de {S}obolev pour un semi-groupe sym\'etrique.
\newblock {\em C. R. Acad. Sci. Paris S\'er. I Math.}, 301(8):411--413, 1985.

\bibitem{MR1230930}
William Beckner.
\newblock Sharp {S}obolev inequalities on the sphere and the
  {M}oser-{T}rudinger inequality.
\newblock {\em Ann. of Math. (2)}, 138(1):213--242, 1993.

\bibitem{BianchiEgnell}
Gabriele Bianchi and Henrik Egnell.
\newblock A note on the {S}obolev inequality.
\newblock {\em J. Funct. Anal.}, 100(1):18--24, 1991.

\bibitem{BV-V}
Marie-Fran{\c{c}}oise Bidaut-V{\'e}ron and Laurent V{\'e}ron.
\newblock Nonlinear elliptic equations on compact {R}iemannian manifolds and
  asymptotics of {E}mden equations.
\newblock {\em Invent. Math.}, 106(3):489--539, 1991.

\bibitem{BDNS}
Matteo Bonforte, Jean Dolbeault, Bruno Nazaret, and Nikita Simonov.
\newblock Stability in {G}agliardo-{N}irenberg-{S}obolev inequalities: Flows,
  regularity and the entropy method.
\newblock {\em Preprint \href{http://arxiv.org/abs/2103.03312}{arXiv:
  2007.03674} and
  \href{https://hal.archives-ouvertes.fr/hal-03160022}{hal-03160022},
  \emph{Memoirs of the AMS}}, to appear.

\bibitem{BrezisLieb}
Haim Brezis and Elliott~H. Lieb.
\newblock Sobolev inequalities with remainder terms.
\newblock {\em J. Funct. Anal.}, 62(1):73--86, 1985.

\bibitem{Brigati_2023}
Giovanni Brigati, Jean Dolbeault, and Nikita Simonov.
\newblock Logarithmic {S}obolev and interpolation inequalities on the sphere:
  Constructive stability results.
\newblock {\em Annales de l'Institut Henri Poincar{\'e} C, Analyse non
  lin{\'e}aire}, (doi: 10.4171/aihpc/106):1--33, November 2023.

\bibitem{https://doi.org/10.48550/arxiv.2302.03926}
Giovanni Brigati, Jean Dolbeault, and Nikita Simonov.
\newblock On {G}aussian interpolation inequalities.
\newblock {\em C.~R. Math. Acad. Sci. Paris}, 362:21--44, 2024.

\bibitem{brigati2024stability}
Giovanni Brigati, Jean Dolbeault, and Nikita Simonov.
\newblock Stability for the logarithmic {S}obolev inequality.
\newblock {\em Preprint \href{http://arxiv.org/abs/2303.12926}{arXiv:
  2303.12926} and
  \href{https://hal.archives-ouvertes.fr/hal-04042046}{hal-04042046}}, 2024.

\bibitem{MR1330619}
Friedemann Brock.
\newblock Continuous {S}teiner-symmetrization.
\newblock {\em Math. Nachr.}, 172:25--48, 1995.

\bibitem{MR1758811}
Friedemann Brock.
\newblock Continuous rearrangement and symmetry of solutions of elliptic
  problems.
\newblock {\em Proc. Indian Acad. Sci. Math. Sci.}, 110(2):157--204, 2000.

\bibitem{CaffarelliGidasSpruck}
Luis~A. Caffarelli, Basilis Gidas, and Joel Spruck.
\newblock Asymptotic symmetry and local behavior of semilinear elliptic
  equations with critical {S}obolev growth.
\newblock {\em Comm. Pure Appl. Math.}, 42(3):271--297, 1989.

\bibitem{MR1132315}
Eric~A. Carlen.
\newblock Superadditivity of {F}isher's information and logarithmic {S}obolev
  inequalities.
\newblock {\em J. Funct. Anal.}, 101(1):194--211, 1991.

\bibitem{MR3024094}
Eric~A. Carlen and Alessio Figalli.
\newblock Stability for a {GNS} inequality and the log-{HLS} inequality, with
  application to the critical mass {K}eller-{S}egel equation.
\newblock {\em Duke Math. J.}, 162(3):579--625, 2013.

\bibitem{CarlenLoss}
Eric~A. Carlen and Michael Loss.
\newblock Extremals of functionals with competing symmetries.
\newblock {\em J. Funct. Anal.}, 88(2):437--456, 1990.

\bibitem{chen2023stability}
Lu~Chen, Guozhen Lu, and Hanli Tang.
\newblock Stability of {H}ardy-{L}ittlewood-{S}obolev inequalities with
  explicit lower bounds.
\newblock {\em Preprint \href{http://arxiv.org/abs/2301.04097}{arXiv:
  2301.04097}}, 2023.

\bibitem{MR3179693}
Shibing Chen, Rupert~L. Frank, and Tobias Weth.
\newblock Remainder terms in the fractional {S}obolev inequality.
\newblock {\em Indiana Univ. Math. J.}, 62(4):1381--1397, 2013.

\bibitem{Christ}
Michael Christ.
\newblock A sharpened {R}iesz-{S}obolev inequality.
\newblock arXiv: \href{https://arxiv.org/abs/1706.02007}{1706.02007}, 2017.

\bibitem{MR2381156}
J{\'e}r{\^o}me Demange.
\newblock Improved {G}agliardo-{N}irenberg-{S}obolev inequalities on manifolds
  with positive curvature.
\newblock {\em J. Funct. Anal.}, 254(3):593--611, 2008.

\bibitem{Dolbeault_2020}
Jean Dolbeault and Maria~J. Esteban.
\newblock Improved interpolation inequalities and stability.
\newblock {\em Advanced Nonlinear Studies}, 20(2):277--291, March 2020.

\bibitem{DEFFL}
Jean Dolbeault, Maria~J. Esteban, Alessio Figalli, Rupert~L. Frank, and Michael
  Loss.
\newblock Sharp stability for {S}obolev and log-{S}obolev inequalities, with
  optimal dimensional dependence.
\newblock {\em Preprint \href{http://arxiv.org/abs/2209.08651}{arXiv:
  2209.08651} and
  \href{https://hal.archives-ouvertes.fr/hal-03780031}{hal-03780031}}, 2023.

\bibitem{MR3177759}
Jean Dolbeault, Maria~J. Esteban, Micha\l{} Kowalczyk, and Michael Loss.
\newblock Improved interpolation inequalities on the sphere.
\newblock {\em Discrete Contin. Dyn. Syst. Ser. S}, 7(4):695--724, 2014.

\bibitem{1302}
Jean Dolbeault, Maria~J. Esteban, and Michael Loss.
\newblock Nonlinear flows and rigidity results on compact manifolds.
\newblock {\em Journal of Functional Analysis}, 267(5):1338 -- 1363, 2014.

\bibitem{MR3640894}
Jean Dolbeault, Maria~J. Esteban, and Michael Loss.
\newblock Interpolation inequalities on the sphere: linear vs. nonlinear flows.
\newblock {\em Ann. Fac. Sci. Toulouse Math. (6)}, 26(2):351--379, 2017.

\bibitem{MR3493423}
Jean Dolbeault and Giuseppe Toscani.
\newblock Stability results for logarithmic {S}obolev and
  {G}agliardo-{N}irenberg inequalities.
\newblock {\em Int. Math. Res. Not. IMRN}, 2016(2):473--498, 2016.

\bibitem{MR4116725}
Ronen Eldan, Joseph Lehec, and Yair Shenfeld.
\newblock Stability of the logarithmic {S}obolev inequality via the
  {F}\"{o}llmer process.
\newblock {\em Ann. Inst. Henri Poincar\'{e} Probab. Stat.}, 56(3):2253--2269,
  2020.

\bibitem{MR3567822}
Max Fathi, Emanuel Indrei, and Michel Ledoux.
\newblock Quantitative logarithmic {S}obolev inequalities and stability
  estimates.
\newblock {\em Discrete Contin. Dyn. Syst.}, 36(12):6835--6853, 2016.

\bibitem{Frank_2022}
Rupert~L. Frank.
\newblock Degenerate stability of some {S}obolev inequalities.
\newblock {\em Annales de l'Institut Henri Poincar{\'e} C, Analyse non
  lin{\'e}aire}, 39(6):1459--1484, June 2022.

\bibitem{GidasNiNirenberg}
Basilis Gidas, Wei~Ming Ni, and Louis Nirenberg.
\newblock Symmetry of positive solutions of nonlinear elliptic equations in
  {$\mathbb R^n$}.
\newblock In {\em Mathematical analysis and applications, {P}art {A}}, volume~7
  of {\em Adv. in Math. Suppl. Stud.}, pages 369--402. Academic Press, New
  York-London, 1981.

\bibitem{MR4475270}
Nathael Gozlan.
\newblock The deficit in the {G}aussian log-{S}obolev inequality and inverse
  {S}antal\'{o} inequalities.
\newblock {\em Int. Math. Res. Not. IMRN}, (17):13396--13446, 2022.

\bibitem{MR4305006}
Emanuel Indrei and Daesung Kim.
\newblock Deficit estimates for the logarithmic {S}obolev inequality.
\newblock {\em Differential Integral Equations}, 34(7-8):437--466, 2021.

\bibitem{arXiv:2210.08482}
Tobias K{\"o}nig.
\newblock On the sharp constant in the {B}ianchi--{E}gnell stability
  inequality.
\newblock {\em Bulletin of the London Mathematical Society}, 55(4):2070--2075,
  April 2023.

\bibitem{arXiv:2211.14185}
Tobias K{\"o}nig.
\newblock Stability for the {S}obolev inequality: existence of a minimizer.
\newblock {\em Preprint \href{http://arxiv.org/abs/2211.14185}{arXiv:
  2211.14185}, J. Eur. Math. Soc.}, to appear.

\bibitem{Lieb}
Elliott~H. Lieb.
\newblock Sharp constants in the {H}ardy-{L}ittlewood-{S}obolev and related
  inequalities.
\newblock {\em Ann. of Math. (2)}, 118(2):349--374, 1983.

\bibitem{MR834360}
P.-L. Lions.
\newblock The concentration-compactness principle in the calculus of
  variations. {T}he limit case. {I}.
\newblock {\em Rev. Mat. Iberoamericana}, 1(1):145--201, 1985.

\bibitem{MR353471}
Henry~P. McKean.
\newblock Geometry of differential space.
\newblock {\em Ann. Probability}, 1:197--206, 1973.

\bibitem{Rodemich}
Eugene~R. Rodemich.
\newblock The {S}obolev inequalities with best possible constants.
\newblock {\em Analysis Seminar Caltech}, 1966.

\bibitem{Rosen}
Gerald Rosen.
\newblock Minimum value for {$c$} in the {S}obolev inequality {$\|\phi\sp{3}\|
  \leq c\,\|\nabla \phi\|\sp{3}$}.
\newblock {\em SIAM J. Appl. Math.}, 21:30--32, 1971.

\bibitem{Talenti}
Giorgio Talenti.
\newblock Best constant in {S}obolev inequality.
\newblock {\em Ann. Mat. Pura Appl. (4)}, 110:353--372, 1976.\hfill\ 

\end{thebibliography}
\bibliographystyle{plain}


\bigskip\begin{center}\rule{2cm}{0.5pt}\end{center}\bigskip

\end{document}